\documentclass[11pt]{article}
\usepackage{amsmath,amsthm,amssymb, color}
\usepackage{hyperref}
\usepackage{cleveref}
\usepackage{natbib}

\usepackage{fullpage}

\usepackage{tikz}
\newcommand*\circled[1]{\tikz[baseline=(char.base)]{
   \node[shape=circle,draw,inner sep=1pt] (char) {#1};}}

\newtheorem{thm}{Theorem}[section]

\newtheorem{lem}[thm]{Lemma}
 
\newtheorem{cor}[thm]{Corollary}

\newtheorem{question}{Question}

\def\F{{\mathbb{F}}}

\def\N{{\mathbb{N}}}

\def\cM{{\cal M}}

\def\_{\,\,\,\,\,}

\def\rank{\textsf{rank}}

\def\VC{\textsf{VC-dim}}
\def\intdeg{\textsf{int-deg}}

\def\symdif{{\triangle}}

\newcommand{\eps}{\epsilon}

\newcommand{\remove}[1]{}

\begin{document}

\title{A Sauer-Shelah-Perles Lemma for Sumsets}

\author{Zeev Dvir\thanks{Department of Computer Science and Department of Mathematics,
Princeton University.
Email: \texttt{zeev.dvir@gmail.com}. Research supported by NSF CAREER award DMS-1451191 and NSF grant CCF-1523816 } \and
Shay Moran\thanks{School of Mathematics, IAS. Email: \texttt{shaymoran1@gmail.com}.
Research supported by the National Science Foundation under agreement No. CCF-1412958 and by the Simons Foundations.}}

\date{}
\maketitle

\begin{abstract}
We show that any family of subsets  $A\subseteq 2^{[n]}$
satisfies $\lvert A\rvert \leq O\bigl(n^{\lceil{d}/{2}\rceil}\bigr)$, 
where $d$ is the VC dimension of $\{S\triangle T \,\vert\, S,T\in A\}$, and $\triangle$ is the symmetric difference operator.
We also observe that replacing $\symdif$ by either $\cup$ or $\cap$ fails to satisfy an analogous statement. 
Our proof is based on the polynomial method; specifically, on an argument due to [Croot, Lev, Pach '17].
\end{abstract} 

%\pagenumbering{arabic}
%%%%%%%%%%%%%%%%%%%%%%%%%%%%%%%%%%%%%%%%%%%%%%%%%5
\section{Introduction}
%%%%%%%%%%%%%%%%%%%%%%%%%%%%%%%%%%%%%%%%%%%%%%%%%5

Let $A \subset 2^{[n]}$ be a family of subsets of an $n$ element set ($[n]$ w.l.o.g). 
The VC dimension of $A$, denoted by $\VC(A)$, is the size of the largest $Y\subseteq [n]$
such that $\{S\cap Y \,\vert\, S\in A\} = 2^{Y}$.
One of the most useful facts about the VC dimension is given by the Sauer-Shelah-Perles Lemma.
\begin{thm}[Sauer-Shelah-Perles Lemma \citep{sauer72density,shelah72comb}]\label{thm:sauer}
Let $d\leq n\in\N$. 
Suppose $A \subset 2^{[n]}$ satisfies $\VC(A) \leq d$. Then $\lvert A\rvert \leq {n \choose \leq d}$.	 
\end{thm}
The Sauer-Shelah-Perles Lemma has numerous applications ranging from model theory, probability theory,
geometry, combinatorics, and various fields in computer science.
A simple-yet-useful corollary of this lemma is that if $\VC(A) \leq d$, and $\star$ is any binary set-operation (e.g.\ $\star\in\{\cap,\cup,\symdif\}$)
then 
\[\bigl\lvert\{ S\star T \,\vert\, S,T\in A\}\bigr\rvert \leq  {n \choose \leq d}\cdot{n \choose \leq d} = O(n^{2d}).\]
This corollary is used, for example, by~\cite{blumer89VC} to derive closure properties for {\it PAC learnability}.
Let~$A{\circled{$\star$}} A$ denote the family $\{ S\star T \,\vert\, S,T\in A\}$.
In this work we explore the converse direction:
Does an upper bound on the VC-dimension $\VC(A\circled{$\star$} A)$ imply an upper bound on~$\lvert A\rvert$? 
It is not hard to see that  $\VC(A)\leq\VC(A\circled{$\star$} A)$ for $\star\in\{\cup,\cap,\symdif\}$,
and therefore, by \Cref{thm:sauer}: $\VC(A\circled{$\star$} A) < d\implies \lvert A\rvert\leq O(n^d)$. 

Our main result quadratically improves this naive bound when $\star$ is symmetric difference:
\begin{thm}\label{thm-main}
Let $d\leq n\in\N$. 
Suppose $A \subset 2^{[n]}$ satisfies $\VC(A \circled{$\symdif$} A) \leq d$. 
Then 
\[\lvert A\rvert \leq 2 {n \choose \leq \lfloor {d}/{2}\rfloor}.\]	 
\end{thm}

We note that \Cref{thm-main} does not hold when $\star\in\{\cup,\cap\}$: pick $d\geq 2$, and set
\[ A = \{ S \subseteq [n] \,\vert\, \lvert S\rvert \leq d\}.\]
Note that $A = A\circled{$\cap$} A $ and therefore $d =\VC(A)= \VC(A\circled{$\cap$} A)$.
However $\lvert A \rvert = {n \choose \leq d} = \Theta(n^d)$, which is not upper bounded by $O(n^{\lceil {d}/{2}\rceil})$.
Picking $A = \{ S \subseteq [n] \,\vert\, \lvert S\rvert \geq n-d\}$ shows that $\cup$ behaves similarly like $\cap$
in this context.

The above examples rules out the analog of \Cref{thm-main} for exactly one of $\cup,\cap$.
This suggests the following open question:
\begin{question}
Let $d\leq n\in\N$. 
Suppose $A \subset 2^{[n]}$ satisfies $\VC(A\circled{$\cap$} A) \leq d$ and $\VC(A\circled{$\cup$} A) \leq d$. 
Is it necessarily the case that $\lvert A\rvert  \leq n^{{d}/{2} + O(1)}$?	 
\end{question}

Another natural question is whether this phenomenon extends to several applications of the symmetric difference operator, for example:
\begin{question}
Does there exist an $\eps < 1/2$ such that for every $d\leq n$ and every $A\subset 2^{[n]}$:
\[ \VC\bigl(A\circled{$\triangle$} A \circled{$\symdif$} A\bigr)\leq d \implies \lvert A\rvert \leq n^{\eps \cdot d + O(1)}?\]
\end{question}
{In Section~\ref{sec:modp} we derive a related statement when $\symdif$ is replaced by addition modulo $p$ for a prime~$p$,
and the VC dimension is replaced by the interpolation degree (which is defined in the next section).}

%The requirement $d\geq 2$ is necessary as for $d=1$, one can pick the chain 
%\[A=\bigl\{\emptyset,\{1\},\{1,2\},\ldots,[n]\bigr\}\]
%that satisfies $A=A\cup A= A\cap A$ and so $\VC(A\cup A) = \VC(A\cap A) = 1$ and $\lvert A\rvert = n+1$.

%Theorem~\ref{thm-main} is particularly interesting for concept classes $A$ that are tight w.r.t the Sauer-Shelah lemma:
%Let $A\subseteq 2^{[n]}$ be a maximum class of VC dimension $d$; 
%that is, $\lvert A\rvert = {n \choose \leq d}$/ 
%Set $\bar A$ be 
%\[\bar A = \{Y\subseteq [n] : Y\in A \text{ or } [n]\setminus Y\in A\}.\]
%It is known that (\red{add references/proof}) $\VC(\bar A) = d+1$ 
% and that $\bar A = 2{n \choose \leq d}$.
%Thus, by Theorem~\ref{thm-main}, $\VC(\bar A\triangle \bar A)\geq 2d$.
%There are many examples of maximum classes $A$ 
%stemming from geometrical and algebraic contexts (\red{add exmaples}).

\subsection{Interpolation degree}
Since our proof method is algebraic, it is convenient to view $A\subset 2^{[n]}$ as a subset of the $n$-dimensional vector space $\F_2^n$ over the field of two elements.  In this setting $A \circled{$\symdif$}A$ is the {\em sumset} of $A$, denoted~$A+A$.
% Since our proof method is algebraic, for the rest of the paper we will think of $A$ in this manner as a subset of $\F_2^n$.

Theorem~\ref{thm-main} will  follow from a stronger statement involving a quantity referred to in some places as the {\em regularity} (as a special case of Castelnuovo-Mumford regularity from algebraic geometry)~\citep{Remscrim16} and in other as the {\em interpolation-degree} \citep{moran16shattered}. We will use the more descriptive interpolation-degree for the rest of this paper. We begin with some preliminary notations and definitions.
 
Let $A \subset \F_2^n$. 
It is a basic fact that for each function $f : \F_2^n \mapsto \F_2$ 
there exists a unique multilinear polynomial $P_f \in \F_2[x_1,\ldots,x_n]$ 
such that $f(a) = P_f(a)$ for all $a\in \F_2^n$ (existence is via simple interpolation and uniqueness follows from dimension counting). 
For a partial function $f : A \mapsto \F_2$ there are  many  (precisely $2^{{2^n - |A|}}$) 
multilinear polynomials whose restriction to $A$ computes $f$. 
Let $\deg_A(f)$ denote the minimal degree of any polynomial whose restriction to $A$ computes $f$. 
We define the {\em interpolation-degree} of $A$, 
denoted $\intdeg(A)$ to be the maximum of $\deg_A(f)$ taken over all functions $f : A \mapsto \F_2$. 
In other words, $\intdeg(A)$ is the smallest $d$ such that 
any function from $A$ to $\F_2$ can be realized by a polynomial of degree at most $d$. 
Clearly, $\intdeg(A)$ is an integer between $0$ and $n$. 
It is also not hard to see that, if $A$ is a proper subset of ${\F_2^n}$ then $\intdeg(A) < n$. 
Our interest in $\intdeg(A)$ comes from the following connection to VC-dimension. 
\begin{lem}[\cite{babai92book,gurvits97linear,smolensky97vc,moran16shattered}]\label{thm:intdegvc}
For $A \subset \F_2^n$ we have $\intdeg(A) \leq \VC(A)$. 
\end{lem}
This Lemma, under various formulations, was proved in several works.
The formulation that appears here can be found in \citep{moran16shattered}.
For completeness, we next sketch the proof:
since the set of all multilinear monomials (also those of degree larger than $\VC(A)$)
span the set of functions $f: A\to\F_2$, 
it suffices to show that any monomial (when seen as an $A\to \F_2$ function) 
can be represented a polynomial of degree at most $d=\VC(A)$. 
The crucial observation is that if $x_S = \pi_{i\in S} x_i$ is a monomial of degree larger than $d$,
then $S$ is not shattered by $A$.
This means that there is a pattern $v:S\to\{0,1\}$ that does not appear
in any of the vectors in $A$ and therefore  
\[ \Pi_{i\in S}(x_i + v_i + 1) =_{A} 0,\]
where ``$=_A$'' means equality as functions over $A$.
Now, expanding this product and rearranging the equation 
yields a representation of $x_S$ as sum of monomials~$x_{S'}$,
where $S' \subset S$, which by induction can also be represented
by polynomials of degree at most $d$.

\Cref{thm:intdegvc} reduces \Cref{thm-main} to the following stronger statement that is proved in the next section.
\begin{thm}\label{thm-intdeg}
Let $d\leq n\in\N$, and let $A \subset \F_2^n$ satisfy $|A| > 2{n \choose \leq \lfloor{d}/{2} \rfloor}$. Then $\intdeg(A+A) > d$.	
\end{thm}

\section{Proof of Theorem~\ref{thm-intdeg}}

The main technical tool will be a lemma of Croot-Lev-Pach \citep{CLP17} that was the main ingredient in the recent solution of the cap-set problem \citep{EG17} and has found many other applications since then (e.g., \citep{Green16, Soly18, DB17, Fox17} to name a few).

\begin{lem}[CLP lemma \citep{CLP17}]\label{lem-clp}
	Let $P \in \F_q[x_1,\ldots,x_n]$ be a polynomial of degree at most $d$ over  any finite field $\F_q$,
	and let $M$ denote the $q^n \times q^n$ matrix with entries $M_{x,y} = P(x+y)$ for $x,y \in \F_q^n$. 
	Then $\rank(M) \leq 2\cdot m_{\lfloor d/2 \rfloor}(q,n),$ 
	where $m_{k}(q,n)$ denotes the number of monomials in $n$ variables $x_1,\ldots,x_n$ 
	such that each variable appears with individual degree at most $q-1$ and the total degree of the monomial is at most $k$.
\end{lem}

Specializing to our setting of $\F_2$ multilinear polynomials, we see that $m_k(2,n) = {n \choose \leq k}$ and so we conclude:

\begin{cor}\label{cor-clp}
	Let $P \in \F_2[x_1,\ldots,x_n]$ be a polynomial of degree at most $d$  and let $M$ be as in Lemma~\ref{lem-clp}. Then $\rank(M) \leq 2 {n \choose \leq \lfloor d/2 \rfloor}$.
\end{cor}

We are now ready to prove Theorem~\ref{thm-intdeg}.

\begin{proof}[Proof of Theorem~\ref{thm-intdeg}]
Suppose $A \subset \F_2^n$ is such that $|A| \geq 2{n \choose \leq \lfloor{d}/{2}\rfloor}$. Let $f : A+A \mapsto \F_2$ be such that $f(\bar 0)=1$, 
where $\bar 0$ is the all zero vector in $\F_2^n$, and $f(a)=0$ for all non-zero $a \in A+A$. 
It suffices to show that $deg_{A+A}(f)\geq \lfloor{d}/{2}\rfloor$
({notice that since $A\neq\emptyset$} it follows that $\bar 0\in A+A$ and so~$f$ is not constantly $0$ on $A+A$).
Let $M$ be the $2^n \times 2^n$ matrix whose rows and columns are indexed by $\F_2^n$ and with entries $M_{x,y} = f(x+y)$. 
By our definition of $f$ we have that the sub-matrix of $M$ whose rows and columns are indexed by $A$ is just the $|A|\times |A|$ identity matrix. This implies $$ \rank(M) \geq |A|.$$ Let $d_f = \deg_{A+A}(f)$ denote the smallest degree of a polynomial whose restriction to $A+A$ computes~$f$. Applying Corollary~\ref{cor-clp} we get that 
\[ \rank(M) \leq 2{ n \choose \leq \lfloor d_f/2 \rfloor}.\]
Combining the two inequalities on $\rank(M)$ and using the bound on the size of $A$ we get that
\[2{n \choose \leq \lfloor{d}/{2}\rfloor} < |A| \leq \rank(M) \leq 2{ n \choose \leq \lfloor d_f/2 \rfloor}, \] 
which implies $\lfloor d/2 \rfloor < \lfloor d_f/2 \rfloor$. This means that $d_f > d$ and so $\intdeg(A+A) > d$.
\end{proof}

\section{Generalization to sums modulo $p$}\label{sec:modp}

In this section we observe that our proof can be generalized to give stronger bounds in the case when we take $p$-fold sums of boolean vectors over $\F_p$. The case proved in the last section corresponds to (two fold) sums modulo 2. For a subset $A \subset \F_p^n$ and a positive integer $k$, we denote by $$k\cdot A = \{ a_1 + \ldots + a_k \,|\, a_i \in A \}$$ the $k$-fold sumset of $A$. To formally define the interpolation degree over $\F_p$ we need to consider, instead of multilinear polynomials, polynomials in which each variable has degree at most $p-1$. We call such polynomials {\em $p$-reduced} polynomials. The space of all $p$-reduced polynomials has dimension $p^n$ and  can uniquely represent any function $f: \F_p^n \mapsto \F_p$. The degree of such a function is defined to be the total degree of the unique $p$-reduced polynomial representing it and can range between 0 and $(p-1)n$. The interpolation degree of a set $A \subset \F_q^n$ is the minimum $d$ such that any function $f : A \mapsto \F_p$ can be represented by a $p$-reduced polynomial of degree at most $d$. To avoid confusion we will denote the interpolation degree over $\F_p^n$ as $\intdeg_p(A)$.

We denote by $\cM_d(p,n)$ the set of monomials in $n$ variables $x_1,\ldots,x_n$ in which each variables has degree at most $p-1$ and the total degree is at most $d$. When $p=2$ we have the closed formula $|\cM_d(2,n)| = {n \choose \leq d}$. When $p >2$ the quantity $|\cM_d(p,n)|$ is a bit more tricky to compute but is known to satisfy certain asymptotic inequalities (e.g., large deviations \citep{LDP15} showing that $\cM_{\delta n}(p,n) \leq 2^{\eps n}$ with $\eps(\delta)$ going to zero with $\delta$). 

The following theorem generalizes Theorem~\ref{thm-main} when $p > 2$.
\begin{thm}\label{thm-psums}
Let $p$ be any prime number and let $A \subset \{0,1\}^n \subset \F_p^n$	be such that $|A| > p\cdot |\cM_{\lfloor d/p \rfloor}(p,n)|$. Then $\intdeg_p(p \cdot A) > d$.
\end{thm}

The proof of the theorem requires the notion of {\em slice-rank} of a tensor which was introduced by Tao in his symmetric interpretation of the proof of the cap-set conjecture \citep{Tao-capset}. By a $k$-fold tensor of dimension $D$ over a field $\F$ we mean a function $T$ mapping  ordered tuples $(j_1,\ldots,j_k) \in [D]^n$ to $\F$.  The {\em slice-rank} of a $k$-fold tensor $T$ is a the smallest integer $R$ such that $T$ can be written as a sum $T = \sum_{i=1}^R T_i$ such that, for every $i \in [R]$ there is some $j_i \in [k]$ so that $T_i(j_1,\ldots,j_k)= A(j_i)B(j_1,\ldots,j_{i-1},j_{i+1},\ldots,j_k)$. In other words, we define the `rank one' tensors to be those in which the dependence on one of the variables is multiplicative (by a function $A(j_i)$) and the rank of a tensor is the smallest number of rank one tensors needed to describe it. For 2-fold tensors (or matrices) this notion coincides with the usual definition of matrix rank. 

The proof of Theorem~\ref{thm-psums} will follow from a combination of two lemmas regarding slice rank. The first lemma generalizes the Croot-Lev-Pach lemma (and proved in an a similar way).

\begin{lem}\label{lem-clpmodp}
Let $f : \F_p^n \mapsto \F_p$ be of degree $d$. Then the $p$-fold $p^n$ dimensional tensor $T: (\F_p^n)^k \mapsto \F_p$ defined by $T(X^1,\ldots,X^p) = f(X^1 + \ldots +X^p)$ has slice rank at most $p\cdot \cM_{\lfloor d/p \rfloor}(p,n)$.
\end{lem}
\begin{proof}
Consider $T$ as a polynomial in $p$ groups of variables $X^i = (x^i_1,\ldots,x^i_n)$ with $i = 1,2,\ldots,p$. Since the degree of $f$ is $d$, the degree of $T$ as a polynomial will also be at most $d$. This means that, in each monomial of $T(X^1,\ldots,X^p) = f(X^1 + \ldots +X^p)$, the degree of at least one group of variables will be at most $\lfloor d/p \rfloor$. Grouping together  monomials according to which group has low degree (if there is more than one group take the one with lowest index) we can represent $T$ as a sum of $p$ tensors, each having rank at most $\cM_{\lfloor d/p \rfloor}(p,n)$. This completes the proof.
\end{proof}

The second lemma needed to prove Theorem~\ref{thm-psums} is due to Tao and shows that the 'diagonal' tensor has full rank.
\begin{lem}[\cite{Tao-capset}]\label{lem-taorank}
Let $\delta(j_1,\ldots,j_k): [D]^k \mapsto \F$ be defined as $\delta(j,j,\ldots,j) = 1$ for all $j$ and is zero otherwise. Then the slice rank of $\delta$ is equal to $D$.
\end{lem}

\begin{proof}[Proof of Theorem~\ref{thm-psums}]
To prove the bound on $\intdeg_p(p\cdot A)$ we describe a function $f : p\cdot A \mapsto \F_p$ that cannot be represented by a low degree polynomial. We take $f$ to be equal to $1$ on the zero vector and zero otherwise. We now consider the tensor $T(X^1,\ldots,X^p) = f(X^1 + \ldots +X^p)$ defined on $A^p$. Notice that, since $A \subset \{0,1\}^n$, the sum of $p$ of them is equal to zero iff all $p$ summands are identical. This implies that $T$ is the diagonal tensor $\delta$ of Lemma~\ref{lem-taorank} and hence has rank equal to $|A|$. On the other hand, if the degree of $f$ (over $p\cdot A$) is at most $d$ then, by Lemma~\ref{lem-clpmodp}, the tensor $T$ has rank at most $p\cdot \cM_{\lfloor d/p \rfloor}(p,n)$. Since we assume that $|A| > p\cdot \cM_{\lfloor d/p \rfloor}(p,n)$ this cannot happen and so $\intdeg_p(pA) > d$.
\end{proof}

\bibliographystyle{plainnat}
\bibliography{refs}

\end{document}